\documentclass[a4paper, 10pt]{article}

\usepackage{cite, amsmath, amssymb, graphicx}
\usepackage[hidelinks]{hyperref}

\usepackage{geometry}
\usepackage{setspace}

\usepackage{amsthm}

\theoremstyle{plain}
\newtheorem{thm}{Theorem}
\newtheorem*{thm*}{Theorem}

\newtheorem{cor}{Corollary}

\newtheorem{lem}{Lemma}
\newtheorem*{lemm}{Lemma}

\newtheorem{prope}{Property}

\theoremstyle{definition}

\newtheorem{defn}{Definition}

\newtheorem*{conj*}{Conjecture}

\newcommand{\inv}{\ensuremath{\texttt{AG}}}

\theoremstyle{remark}

\theoremstyle{definition}

\usepackage{graphicx}


\newcommand{\acknowledgment}[1]{\vspace{5mm}\singlespacing
	{\noindent\textbf{\textit{Acknowledgment\/}:} #1}
}

\usepackage[margin=1cm,%
			font=footnotesize,%
			format=hang,%
			labelsep=period,%
			labelfont=bf]{caption}

\usepackage[latin1]{inputenc}
\usepackage{color}
\usepackage{color, colortbl}
\usepackage[english]{babel}
\usepackage{amsthm}
\usepackage{amsfonts}
\usepackage{enumitem}
\usepackage{enumerate}
\usepackage{color}
\usepackage{cite}
\usepackage{tikz}
\usepackage{hyperref}
\usepackage{hhline}

\title{Extremal Chemical  Graphs for the Arithmetic-Geometric Index}
\author{Alain Hertz$^{a,}$\footnote{Corresponding author: alain.hertz@gerad.ca.}\ , S\'ebastien Bonte$^b$, Gauvain Devillez$^{b,c}$,\\Valentin Dusollier$^b$, Hadrien M\'elot$^b$, David Schindl$^d$\\[2ex]
\small $^a$Department of Mathematics and Industrial Engineering,\\[-3pt]
\small	Polytechnique Montr\'eal - Gerad, Montr\'eal, Canada\\
\small	$^b$Computer Science Department - Algorithms Lab\\[-3pt]
\small	University of Mons, Mons, Belgium\\
\small	$^c$Faculty of Mathematics and Physics\\[-3pt]
\small	University of Ljubljana, Ljubljana, Slovenia\\
\small	$^d$Haute \'ecole de gestion de Gen\`eve\\[-3pt]
\small	University of Applied Sciences Western Switzerland, Gen\`eve, Switzerland\\
}

\date{\today}

\begin{document}

\maketitle

\begin{abstract}
The arithmetic-geometric index is a newly proposed degree-based graph invariant in mathematical chemistry.  We give a sharp upper bound on the value of this invariant for connected chemical graphs of given order and size and characterize the connected chemical graphs that reach the bound.
We also prove that the removal of the constraint that extremal chemical graphs must be connected does not allow to increase the upper bound.
\end{abstract}

\onehalfspacing

\section{Introduction}

In mathematical chemistry, and more specifically in chemical graph theory, topological indices are values that characterize a graph representing a molecule. This value is correlated with some properties of said molecule. One largely studied class of such indices is composed of degree-based invariants. A well known example of these is the Randi\'c index, proposed by Randi\'c~\cite{BIBLIO13} in 1975. Following the interest in this index, many others were later introduced. Of interest here is the \emph{arithmetic-geometric index} proposed in 2015 by Shegehalli and Kanabur \cite{BIBLIO8}. Let $G$ be a graph with edge set $E$ and let $d_u$ and $d_v$ be the degrees of the endpoints of an edge $uv\in E$. The arithmetic-geometric index $\inv(G)$ of $G$, is defined as
\[
\inv(G) = \sum_{uv \in E} \frac{d_u + d_v}{2 \sqrt{d_ud_v}}.
\]
The summand in the above formula is the ratio between the arithmetic and geometric means of $d_u$ and $d_v$. If we replace each summand  by its inverse, we obtain another graph invariant known as the 
geometric-arithmetic index introduced in 2009 by Vuki{\v{c}}evi{\'c} and Furtula~\cite{BIBLIO14} and studied for example in \cite{R16,B19,A20,Das}.
Several papers have focused on properties of the arithmetic-geometric index. For example, Shegehalli and Kanabur \cite{BIBLIO8,BIBLIO9} give the value of this index for some families of graphs. Lower bounds for graphs of fixed size (i.e., number of edges) are provided in \cite{BIBLIO2,BIBLIO4,BIBLIO5}, while an upper bound for graphs with fixed size and order (i.e., number of vertices) is established in \cite{BIBLIO5}. Upper bounds for graphs of fixed size and fixed minimum and maximum degrees are given in  \cite{BIBLIO2,BIBLIO3,BIBLIO4}. 
The maximum value of the arithmetic-geometric index of graphs of fixed order is known for  
unicyclic graphs \cite{BIBLIO6},
bicyclic graphs \cite{BIBLIO7}, bipartite graphs and trees \cite{BIBLIO5}. In addition,  Vujo{\v{s}}evi{\'c} et al.  \cite{BIBLIO5} have  characterized the  chemical trees with maximum arithmetic-geometric index value. The relationship between the arithmetic-geometric index and other topological indices is studied, for example in \cite{BIBLIO1,BIBLIO2,BIBLIO3,BIBLIO4,BIBLIO10,BIBLIO11,BIBLIO12}.

In this paper, we prove the following upper bound on the  value of the arithmetic index of a connected chemical graphs $G$ of order $n$ and size $m$:
$$\inv(G)\leq\frac{2n+5m}{6}+\left\{\begin{array}{ll}0& \mbox{\emph{if} } 2m-n\equiv 0 \bmod{3},\\
\frac{3}{\sqrt{2}}-\frac{13}{6}&\mbox{\emph{if} } 2m-n\equiv 1 \bmod{3},\\\frac{21}{4\sqrt{3}}-\frac{37}{12}&\mbox{\emph{if} } 2m-n\equiv 2 \bmod{3}.\end{array}\right.$$

We show that with the exception of 22 $(n,m)$ pairs, the bound is sharp, and  we characterize the connected chemical graphs of order $n$ and size $m\geq n-1$ that reach the bound. We also prove that no better value can be obtained by removing the constraint that the graph must be connected. Note that for $m=n-1$, this gives a characterization of extremal chemical trees of fixed order $n$. While such a characterization is given in \cite{BIBLIO5}, we show that their result is not valid for 7 values of $n$.

In the next section we fix some notations, while Section \ref{sec:preliminaries} is devoted to observations that will motivate our characterization of connected chemical graphs with maximum arithmetic-geometric index value. Lemmas are proved in Section \ref{sec:tools} and then used in Section \ref{sec:MainTheorem} to prove the main theorem. 

\section{Notations}

For basic notions of graph theory that are not defined here, we refer to Diestel \cite{Diestel00}. Let $G=(V,E)$ be a simple undirected graph. The \emph{order} $n = |V|$ of $G$ is its number of vertices and the \emph{size} $m=|E|$ of $G$ is its number of edges. We write $G \simeq H$ if $G$ and $H$ are isomorphic. The \emph{degree} of $v$, denoted $d_v$ is the number of edges incident to $v$, and we say that $v$ is \emph{isolated} if $d_v=0$.

A \emph{chemical graph} is a graph whose vertices have degree at most 4. The arithmetic-geometric index $\inv(G)$ of a graph $G$ can be seen as a sum of costs on the edges of $G$. In particular, if we deal with chemical graphs, there is a limited number of possible values for the costs since they are computed from the degrees of the endpoints of the edges. Let $c_{i,j}= \frac{i + j}{2 \sqrt{ij}}$ be the cost of an edge with endpoints of degree $i$ and $j$. The  
$4 \times 4$ cost matrix  $C_{\inv}$ associated with the arithmetic-geometric index of chemical graphs is
\begin{equation*} 
C_{\inv} = 
\begin{pmatrix}
1 & \frac{3}{2\sqrt{2}} & \frac{2}{\sqrt{3}} & \frac{5}{4}\\[1.2ex]
\frac{3}{2\sqrt{2}} & 1 & \frac{5}{2\sqrt{6}} & \frac{3}{2\sqrt{2}} \\[1.2ex]
\frac{2}{\sqrt{3}} & \frac{5}{2\sqrt{6}} & 1 & \frac{7}{4\sqrt{3}} \\[1.2ex]
\frac{5}{4} & \frac{3}{2\sqrt{2}}  & \frac{7}{4\sqrt{3}}  & 1\\[1.2ex]
\end{pmatrix}
\approx
\begin{pmatrix}
1.0000  & 1.0607  & 1.1547 &  1.2500\\
1.0607  & 1.0000  & 1.0206  & 1.0607\\ 
1.1547  & 1.0206  & 1.0000  & 1.0104\\ 
1.2500  & 1.0607 &  1.0104  & 1.0000\\ 
\end{pmatrix}.
\end{equation*}

For a chemical graph $G$, let $n_i(G)$ ($i=0,\ldots,4$) be the number of vertices of degree $i$ and let $x_{i,j}(G)$ ($1\leq i\leq j\leq 4$) be the number of edges with extremities of degrees $i$ and $j$ in $G$. Then,  since $C_{\inv}$ is symmetric, we have
\[
\inv(G) = \sum_{1\leq i\leq j\leq 4}c_{i,j}\,x_{i,j}(G).
\]

In what follows, we say that a chemical graph $G$ is \emph{extremal} if $\inv(G)\geq \inv(G')$ for all chemical graphs $G'$ with the same order and the same size as $G$.

\section{Preliminaries}\label{sec:preliminaries}

We begin this section with the definition of a class of chemical graphs which, as we will see, contains most of the extremal chemical graphs of order $n$ and size $m$.

\begin{defn}
	$\mathcal{G}_{n,m}$ is the set of chemical graphs of order $n$ and size $m$, and such that $n_0(G)=0$, $n_2(G)+n_3(G)\leq 1$ and all edges have at least one endpoint of degree 4.
\end{defn}

For example, using \emph{Nauty geng} \cite{MCKAY201494} or \emph{PHOEG} \cite{PHOEG} to enumerate all chemical graphs having order $n$ and size $m$, it can be observed that there is only one graph $G_{n,m}$ in $\mathcal{G}_{n,m}$ for $(n,m)=(5,4)$, $(6,10)$, $ (7,7)$, $ (7,9)$, $ (8,9)$, $ (9,8)$, $ (9,9)$ and $(11,11)$, and there are two graphs in $G_{12,11}$, one connected and one non-connected (see also Table \ref{TableConNonCon} at the end of the paper). These graphs are shown in Figure~\ref{figExemples} and they were chosen because they will appear in the proofs of the next sections.

\begin{figure}[!htb]
	\centering\includegraphics[scale=0.34]{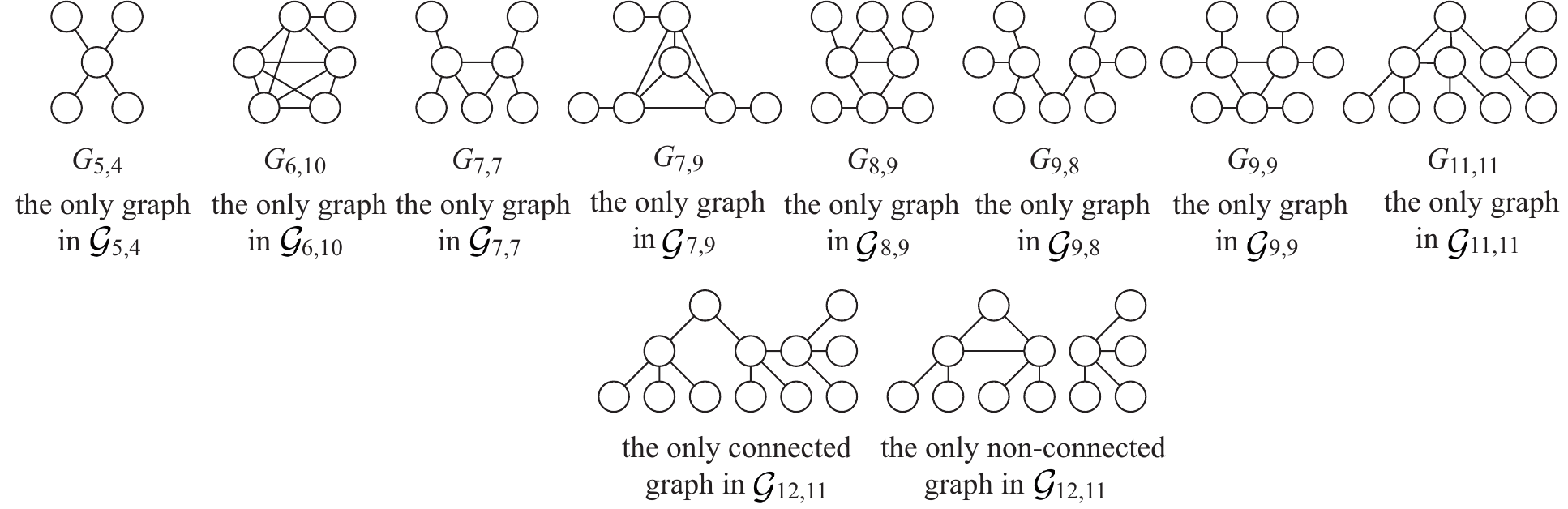}
	\vspace{-0.2cm}\caption{Examples of graphs in $\mathcal{G}_{n,m}$ for some pairs $(n,m)$.}
	\label{figExemples}
\end{figure}

If all edges of a chemical graph $G$ have at least one endpoint of degree 4, then $\inv(G)=c_{1,4}n_1(G)+2c_{2,4}n_2(G)+3c_{3,4}n_3(G)+(m-n_1(G)-2n_2(G)-3n_3(G))$ and since $2m=n_1(G)+2n_2(G)+3n_3(G)+4n_4(G)$, we have 
\begin{eqnarray*}
	\inv(G)&=&c_{1,4}n_1(G)+2c_{2,4}n_2(G)+3c_{3,4}n_3(G)+4n_4(G)-m\\[-3pt]
	& =& c_{1,4}n_1(G)+2c_{2,4}n_2(G)+3c_{3,4}n_3(G)+4n_4(G)\\
	&&-\tfrac{1}{2}(n_1(G)+2n_2(G)+3n_3(G)+4n_4(G))\\[-3pt]
	& =&\tfrac{3}{4}n_1(G)+(\tfrac{3}{\sqrt{2}}-1)n_2(G)+(\tfrac{21}{4\sqrt{3}}-\tfrac{3}{2})n_3(G)+2n_4(G).
\end{eqnarray*}
For a pair $(n,m)$ of integers, let  $T_{n,m}$ be the set of quadruplets $(t_1,t_2,t_3,t_4)$ of positive integers such that $\sum_{i=1}^4t_i=n$ and $\sum_{i=1}^4it_i=2m$. Hence, we have $2m-n=t_2+2t_3+3t_4$. Note that $T_{1,0}=\emptyset$ since $\sum_{i=1}^4t_i\leq \sum_{i=1}^4it_i$. For $(t_1,t_2,t_3,t_4)\in T_{n,m}$, let $f(t_1,t_2,t_3,t_4)$ be defined as
$$f(t_1,t_2,t_3,t_4)=\tfrac{3}{4}t_1+(\tfrac{3}{\sqrt{2}}-1)t_2+(\tfrac{21}{4\sqrt{3}}-\tfrac{3}{2})t_3+2t_4.$$
Clearly, if $G$ is a chemical graph of order $n$ and size $m$, with no isolated vertex and in which all edges have at least one endpoint of degree 4, then $(n_1(G),n_2(G),n_3(G),n_4(G))$ belongs to $T_{n,m}$ and we have observed above that  
$$\inv(G)=f(n_1(G),n_2(G),n_3(G),n_4(G)).$$

\noindent Let $(t_1,t_2,t_3,t_4)$ be a quadruplet in $T_{n,m}$ with $t_2+t_3\leq 1$:
\begin{itemize}\itemsep=-3pt
	\item if $t_2=1$ then $t_3=0$, which means that $2m-n=3t_4+1$ and $2m=t_1+2+4(n-t_1-1)$, or equivalently,  $t_1=\frac{4n-2m-2}{3}$;
	\item if $t_3=1$ then $t_2=0$, which means that $2m-n=3t_4+2$ and $2m=t_1+3+4(n-t_1-1)$, or equivalently,  $t_1=\frac{4n-2m-1}{3}$;
	\item if $t_2=t_3=0$, then $2m-n=3t_4$ and $2m=t_1+4(n-t_1)$, or equivalently, $t_1=\frac{4n-2m}{3}$.
\end{itemize}
Hence in all cases, we deduce the following property.
\begin{prope} \label{prop1}
	If $(t_1,t_2,t_3,t_4)$ is a quadruplet in $T_{n,m}$ with $t_2+t_3\leq 1$, then
	\begin{itemize}\itemsep=0pt
		\item $t_1=\lfloor \frac{4n-2m}{3}\rfloor$
		\item $t_2=\left\{\begin{array}{ll}1& \mbox{\emph{if} } 2m-n\equiv 1 \bmod{3}\\
		0&\mbox{\emph{otherwise}}\end{array}\right.$
		\item $t_3=\left\{\begin{array}{ll}1& \mbox{\emph{if} } 2m-n\equiv 2 \bmod{3}\\
		0&\mbox{\emph{otherwise}}\end{array}\right.$
		\item $t_4=\lfloor \frac{2m-n}{3}\rfloor$.
	\end{itemize}
\end{prope}

\begin{cor}\label{cor1}
	There is at most one quadruplet $(t_1,t_2,t_3,t_4)$ in $T_{n,m}$ with $t_2+t_3\leq 1$.
\end{cor}

\begin{cor}\label{cor2}
	If $G$ is a graph in $\mathcal{G}_{n,m}$, then
	the quadruplet $(t_1,t_2,t_3,t_4)=(n_1(G),n_2(G),n_3(G),n_4(G))$ is the unique one in $T_{n,m}$ with $t_2+t_3\leq 1$.
\end{cor}
\begin{proof} Let $G$ be a graph in $\mathcal{G}_{n,m}$. Then  $\sum_{i=1}^4n_i(G)=n$ and  $\sum_{i=1}^4in_i(G)=2m$, which means that $(t_1,t_2,t_3,t_4)=(n_1(G),n_2(G),n_3(G),n_4(G))$ is a quadruplet in $T_{n,m}$ with $t_2+t_3\leq 1$. By Corollary \ref{cor1}, it is unique.
\end{proof} 

Some connected extremal chemical graphs have all edges with at least one endpoint of degree 4, but have $n_0(G)>0$ or $n_2(G)+n_3(G)>1$. Seven examples are shown in Figure \ref{fign2n3} and we will prove that there are no other ones.

\begin{figure}[!htb]
	\centering\includegraphics[scale=0.40]{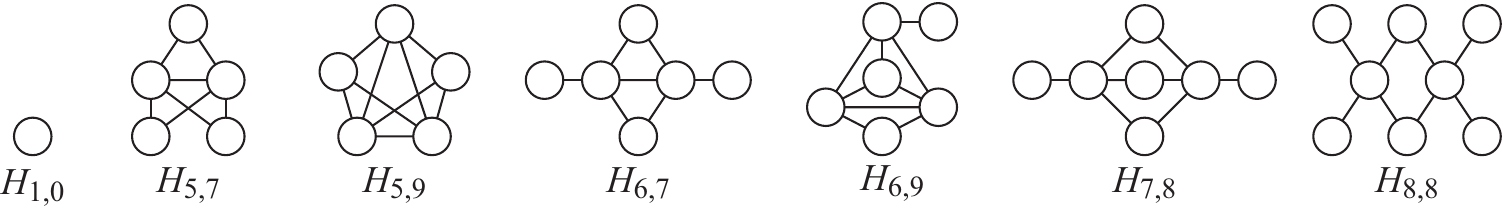}
	\vspace{-0.2cm}\caption{Seven connected  extremal chemical graphs with all edges having at least one endpoint of degree 4 and with $n_0(G)>0$ or $n_2(G)+n_3(G)>1$.}
	\label{fign2n3}
\end{figure}

Also, some connected extremal chemical graphs have at least one edge with no endpoint of degree 4. Fifteen examples are shown in Figure \ref{figxij} and we will prove that there are no other ones.

\begin{figure}[!htb]
	\centering\includegraphics[scale=0.43]{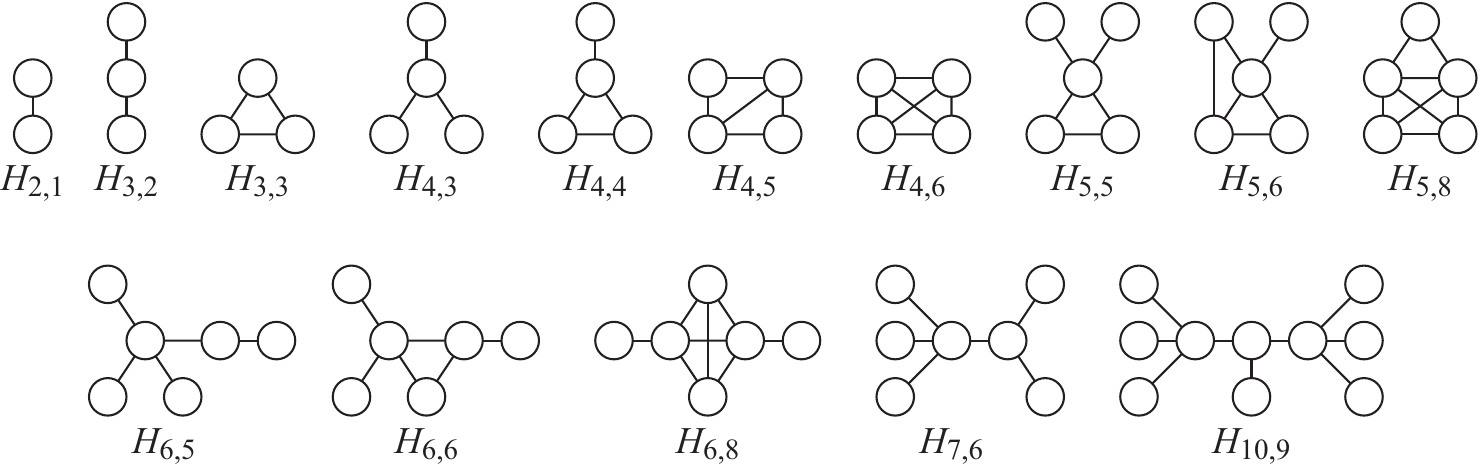}
	\vspace{-0.2cm}\caption{Fifteen connected extremal chemical graphs with at least one edge having no endpoint of degree~4.}
	\label{figxij}
\end{figure}

For each pair $(n,m)$ such that $H_{n,m}$ appears in Figure  \ref{fign2n3} or  \ref{figxij}, we can enumerate all  chemical graphs having order $n$ and size $m$, using again \emph{Nauty geng} \cite{MCKAY201494} or \emph{PHOEG} \cite{PHOEG}. Table  \ref{geng} gives the number of such graphs and it is therefore easy to verify that the following property holds.

\begin{prope}\label{22exceptions}
	The 22 graphs in Figures \ref{fign2n3} and  \ref{figxij} are the only extremal graphs of their order and size.
\end{prope}

\setlength{\tabcolsep}{3pt} 
\begin{table}[!htb]\footnotesize
	\centering
	\caption{Number $N$ of chemical graphs for some orders $n$ and sizes $m$.}
	\label{geng}
	\begin{tabular}{|c||c|c|c|c|c|c|c|c|c|c|c|c|c|c|c|c|c|c|c|c|c|c|}
		\hline
		$n$ &1&2&3&3&4&4&4&4&5&5&5&5&5&6&6&6&6&6&7&7&8&10\\
		$m$ &0&1&2&3&3&4&5&6&5&6&7&8&9&5&6&7&8&9&6&8&8&9\\
		$N$ &1&1&1&1&3&2&1&1&6&6&4&2&1&14&20&22&20&15&38&82&188&883\\
		\hline
	\end{tabular}
\end{table}

The next property relates quadruplets in $T_{n,m}$ with connected graphs in $\mathcal{G}_{n,m}$. Note that a connected chemical graph of order $n$ has $m$ edges, with $n-1\leq m\leq \min \{2n,\tfrac{n(n-1)}{2}\}
$.

\begin{prope}\label{prop:existence}
	Let $n$ and $m$ be two positive integers such that $n-1\leq m\leq \min \{2n,\tfrac{n(n-1)}{2}\}$ and $(n,m)$ is not one of the 22 pairs for which there is a graph $H_{n,m}$ in Figure  \ref{fign2n3} or \ref{figxij}. If $(t_1,t_2,t_3,t_4)\in T_{n,m}$ and $t_2+t_3\leq 1$, then $\mathcal{G}_{n,m}$ contains at least one connected graph $G$ with $n_i(G)=t_i$ ($i=1,2,3,4$).
\end{prope}
\begin{proof}
	Consider a pair $(n,m)$ of positive integers such that $n-1\leq m\leq \min \{2n,\tfrac{n(n-1)}{2}\}$ and let $(t_1,t_2,t_3,t_4)$ be any quadruplet in $T_{n,m}$. Note  that $n\geq 2$ since $T_{1,0}=\emptyset$. According to Property \ref{prop1} and Corollary \ref{cor2}, a graph in $\mathcal{G}_{n,m}$ with $n_i(G)=t_i$ ($i=1,2,3,4$) must have
	\begin{itemize}
		\itemsep=0pt
		\item $n_1(G)=\lfloor \frac{4n-2m}{3}\rfloor$
		\item $n_2(G)=\left\{\begin{array}{ll}1& \mbox{if } 2m-n\equiv 1 \bmod{3}\\
		0&\mbox{otherwise}\end{array}\right.$
		\item $n_3(G)=\left\{\begin{array}{ll}1& \mbox{if } 2m-n\equiv 2 \bmod{3}\\
		0&\mbox{otherwise}\end{array}\right.$
		\item $n_4(G)=\lfloor \frac{2m-n}{3}\rfloor$.
	\end{itemize}
	
	\noindent Moreover, in order to impose that all edges in $G$ have at least one endpoint of degree 4, we must have 
	\begin{itemize}
		\itemsep=0pt
		\item $x_{1,1}(G)=x_{1,2}(G)=x_{1,3}(G)=x_{2,2}(G)=x_{2,3}(G)=x_{3,3}(G)=0$, 
		\item $x_{1,4}(G){=}n_1(G)$, $x_{2,4}(G){=}2n_2(G)$,  $x_{3,4}(G){=}3n_3(G)$, $x_{4,4}(G){=}m{-}n_1(G){-}2n_2(G){-}3n_3(G)$.
	\end{itemize}
	\noindent The following algorithm builds such a connected graph $G$, where $V_i$ ($i=1, \ldots,4$) is the set of vertices of degree $i$ in $G$. It is illustrated in Figure \ref{fig2}.
	\begin{itemize}\itemsep=1pt
		\item[1.] Start from a graph of order $n$ and size $0$. Put $n_i(G)$ vertices in $V_i$, $i=1,\ldots,4$;
		\item[2.] if $2m-n\equiv 1 \bmod{3}$ then connect the vertex in $V_2$ to 2 vertices in $V_4$; 
		\item[3.] if $2m-n\equiv 2 \bmod{3}$ then connect the vertex in $V_3$ to 3 vertices in $V_4$;
		\item[4.] add $n_4(G){-}n_2(G){-}2n_3(G){-}1$ edges that link pairs of vertices in $V_4$ so that the graph induced by $V_2\cup V_3\cup V_4$ is a tree;
		\item[5.] add $m-n_1(G)-n_2(G)-n_3(G)-n_4(G)+1$ edges that link pairs of vertices in $V_4$ so that no vertex in $V_4$ is incident to more than 4 edges;
		\item[6.] add edges linking each vertex of $V_1$ to a vertex of $V_4$ so that every vertex in $V_4$ has degree 4.
	\end{itemize}
	Steps 2 and 3 add the required number of edges with one endpoint of degree 4 and the other  of degree 2 or 3, and Step 6 adds the required number of edges with one endpoint of degree 4 and the other of degree 1.
	Step 4 adds $n_4(G){-}n_2(G){-}2n_3(G){-}1$ edges linking pairs of vertices in $V_4$, while Step 5 adds 
	$m-n_1(G)-n_2(G)-n_3(G)-n_4(G)+1$ such edges. In total we will therefore have $m-n_1(G)-2n_2(G)-3n_3(G)$ edges linking pairs of vertices of $V_4$, which is the required number of edges with both endpoints of degree 4. It remains to prove that such a construction is always possible. For this purpose, the following constraints must be satisfied : 
	\begin{itemize}\itemsep=-2pt\item $2n_2(G)\leq n_4(G)$ and $3n_3(G)\leq n_4(G)$ to ensure that Steps 2 and 3 can be performed;
		\item $m{-}n_1(G){-}2n_2(G){-}3n_3(G)\leq \tfrac{n_4(G)(n_4(G){-}1)}{2}$ to avoid creating parallel edges in Steps 4 and 5;
		\item $n_1(G)\leq 4n_4(G)$, to ensure that Step 6 can be performed.
	\end{itemize}
	
	It is easy to check that these conditions are satisfied for all pairs $(n,m)$ with $n\leq 13$ and $n-1\leq m \leq \min \{2n,\tfrac{n(n-1)}{2}\}$, except for the 22 pairs for which we have a graph $H_{n,m}$ in Figures  \ref{fign2n3} or \ref{figxij}. So assume $n\geq 14$. We then have
	
	\begin{itemize}
		\item $n_4(G)\geq \frac{2m-n-2}{3}\geq\frac{n-4}{3}>3\geq \max\{2n_2(G),3n_3(G)\}$ (since $n_1(G)+n_2(G)\leq 1$). 
	\end{itemize}		
	\vspace{-0.5cm}\begin{align*}
		 \hspace{-0.0cm}\bullet\;\;\frac{n_4(G)(n_4(G){-}1)}{2}{-}x_{4,4}(G)&\geq \frac{1}{2}
		\Biggl(\frac{2m{-}n{-}2}{3}\left(\frac{2m{-}n{-}2}{3}{-}1\right)\Biggr)\\
		&{-}\left(m{-}\frac{4n{-}2m{-}2}{3}\right)\\
		&= \frac{4m(m-n-11)+n^2+31n-2}{18}\\
		&\geq \frac{4(n{-}1)(n{-}1{-}n{-}11){+}n^2{+}31n{-}2}{18}\\
		&=  \frac{n^2-17n+46}{18}>0.
		\end{align*}
	\begin{itemize}
		\item $4n_4(G)- n_1(G)\geq \frac{4(2m-n)}{3}-\frac{4n-2m+2}{3}=\frac{10m-8n-2}{3}\geq \frac{10(n-1)-8n-2}{3}=\frac{2n}{3}-4>0$. 
	\end{itemize}

	\vspace{0.2cm}	In summary, $G$ has the right number of edges of each type and thanks to Step 4, it is connected.
\end{proof}

The algorithm in the above proof is illustrated in Figure \ref{fig2} for $n=m=17$. In such a case, we have $n_1(G)=11$, $n_2(G)=0$, $n_3(G)=1$, $n_4(G)=5$ and $x_{4,4}(G)=3$. In Step 1, we have represented the vertices of $V_1$ with the white color, while the vertex in $V_3$ is grey and the vertices in $V_4$ are black. Step 3 links the grey vertex to 3 black vertices. Step 4 adds 2 edges between black vertices. Step 5 adds the last edge between two black vertices and Step 6 adds the edges between the white and the black vertices. Notice that Step 4 was crucial to obtain a connected graph. Indeed another set of 3 edges linking black vertices could have produced a non-connected graph in $\mathcal{G}_{n,m}$ as illustrated at the bottom right of Figure~\ref{fig2}.

\begin{figure}[!htb]
	\centering\includegraphics[scale=0.32]{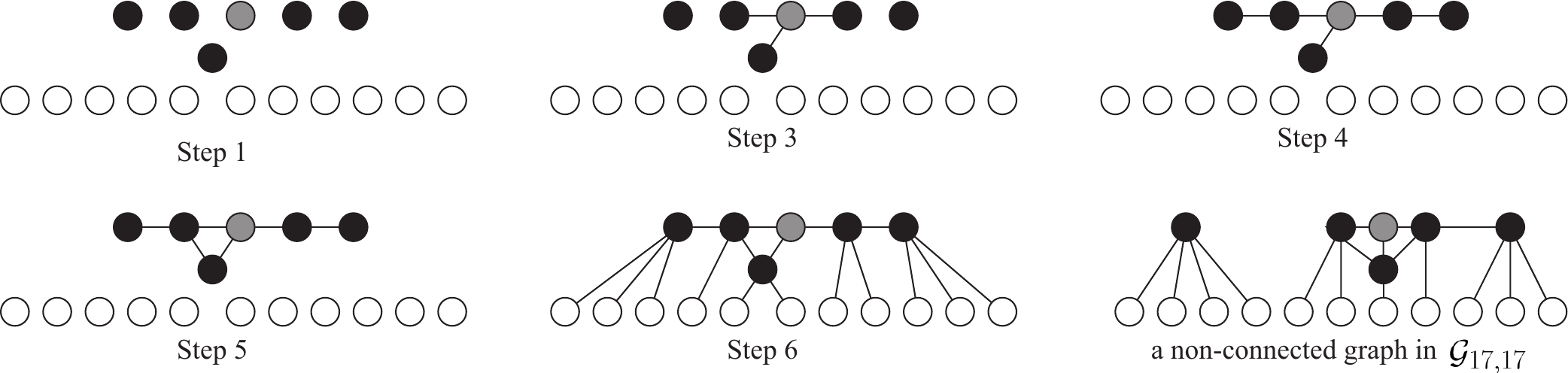}
	\caption{Illustration of the algorithm in the proof of Theorem \ref{thm1}}
	\label{fig2}
\end{figure}

\noindent The main objective of this paper is to prove the following theorem.

\begin{thm}\label{thm1}
	Let $G$ be a connected chemical graph of order $n$ and size $m$. If $G$ is extremal, then either $G$ is one of the 22 graphs $H_{n,m}$ of Figures \ref{fign2n3} and \ref{figxij}, or $G$ belongs to $\mathcal{G}_{n,m}$.
\end{thm}

\noindent To prove this theorem, we need some tools which are given in the next section.
\section{Tools}\label{sec:tools}

\begin{lem}\label{lem_rotation}
	Let $G$ be a connected extremal chemical graph. Assume that $G$ has a vertex $u$ of degree 2 where $v$ and $w$ are its two neighbors.
	\vspace{-6pt}\begin{itemize}[itemsep=0pt]
		\item[(a)] If $v$ and $w$ are nonadjacent, then none of them has degree 3. 
		\item[(b)] If $v$ and $w$ are adjacent, $d_v{\geq} 3$ and $d_w{\leq} 3$, then no vertex nonadjacent to $w$ has degree 2 or 3.
	\end{itemize}
\end{lem}
\begin{proof}$ $
	\vspace{-6pt}\begin{itemize}[itemsep=0pt]
		\item[(a)] Assume that $v$ and $w$ are nonadjacent, and that one of them, say $v$, has degree 3. Let $G'$ be the graph obtained from $G$ by replacing $uw$ with $vw$. Then, with $i=d_u$ and $j=d_x$, where $x$ is any neighbor of $v$ other than $u$, we have
		\begin{align*}
		\inv(G') {-} \inv(G) &\geq c_{1,4} -c_{2,3} + \min_{i = 1, \ldots 4} (c_{4,i} - c_{2,i}) +2 \min_{j = 1, \ldots 4} (c_{4,j} {-} c_{3,j}) \\&\approx 0.1479 > 0.\end{align*}
		\vspace{-20pt}\item[(b)]  Assume that $v$ and $w$ are adjacent with $d_v\geq 3$ and $d_w\leq 3$, and let $x$ be a vertex nonadjacent to $w$ such that $d_x=2$ or 3. Let $G'$ be the graph obtained from $G$ by replacing $uw$ with $xw$. Then, with $i=d_v$, $j=d_w$, $k=d_x$ and $\ell=d_y$, where $y$ is any neighbor of $x$, we have 
		
		\vspace{-25pt}\begin{align*}
			\inv(G') {-} \inv(G) \;{\geq} & \;\min_{i = 3, 4}(c_{1,i} -c_{2,i})\\
			& \;{+} 
			\min_{k = 2, 3}{\bigg(}\min_{j = 2, 3}(c_{j,k{+}1}{-}c_{j,2}){+}k\min_{\ell = 1, \ldots, 4} (c_{\ell,k{+}1}{-}c_{\ell,k}){\bigg)}\\          
		\;	{\approx}&\;  0.0128 >0.
		\end{align*}
	\end{itemize}
	In both cases, $G'$ is connected and  $\inv(G')>\inv(G)$, which means that $G$ is not an extremal, connected chemical graph, a contradiction.
\end{proof}

\begin{lem}\label{lem_2opt}
	A connected extremal chemical graph does not contain a chain $v_1,v_2,\ldots,v_r$ as partial subgraph with $v_1$ nonadjacent to $v_{r-1}$ and $v_2$ nonadjacent to $v_r$ in $G$ and with $d_{v_1}<d_{v_r}$, $d_{v_2}\leq 3$, and $d_{v_{r-1}}=4$. 
\end{lem}

\begin{proof}
	Let $G'$ be the graph obtained from $G$ by replacing the edges $v_1,v_2$ and $v_{r-1},v_r$ by $v_1,v_{r-1}$ and $v_2,v_r$. Then, with $d_{v_1}=i$, $d_{v_2}=j$ and $d_{v_r}=k$, we have
	$$
	\inv(G') - \inv(G) \geq \min_{i = 1,2,3} \;\;\min_{j = 2,3}\;\;\min_{k = i+1,\ldots,4}(c_{i,4} + c_{j,k} - c_{i,j}-c_{4,k}) \approx 0.0207 > 0.$$
	
	Since $G'$ is connected and  $\inv(G')>\inv(G)$, this means that $G$ is not an extremal, connected chemical graph, a contradiction.
\end{proof}

The next lemmas have a label $(i,j)$ with $i<j$ to indicate that they state that a connected extremal chemical graph $G$ has $x_{i,j}(G)=0$, with a few exceptions.

\begin{lemm}[\bf 1,1]
	The only connected extremal chemical graph $G$ with $x_{1,1}(G)>0$ is $H_{2,1}$.
\end{lemm}

\begin{proof}
	Let $G$ be a connected extremal  chemical graph with two adjacent vertices of degree 1. Since $G$ is connected, it does not contain any other vertex, which means that $G \simeq H_{2,1}$. 
\end{proof}

\begin{lemm}[\bf 2,2]
	The only connected extremal  chemical graphs $G$ with\linebreak$x_{2,2}(G){>}0$ are $H_{3,3}$, $H_{4,4}$ and $H_{5,5}$.
\end{lemm}

\begin{proof}
	Let $u$ and $v$ be two adjacent vertices of degree 2 in a connected extremal  chemical graph $G$.
	\begin{itemize}
		\item Assume that $u$ and $v$ have a  common neighbor $w$. If $w$ has degree 2, then $G \simeq H_{3,3}$. So suppose $w$ has degree at least 3. We know from Lemma \ref{lem_rotation}(b) that  all other vertices in the graph have degree 1 or 4. If they have all degree 1, then $G \simeq H_{4,4}$ (if $w$ has degree 3) or $G \simeq H_{5,5}$ (if $w$ has degree 4). So assume $w$ is adjacent to a vertex $x$ of degree 4. If $x$ is adjacent to a vertex $y \neq w$ of degree 4, then Lemma \ref{lem_2opt} with the partial chain $u,v,w,x,y$ contradicts the fact that $G$ is a connected extremal  chemical graph.
		Hence, all neighbors $y\neq w$ of $x$ have degree 1. Similarly, if $w$ has a second neighbor $z$ of degree 4, then all neighbors of $z$, except $w$, have degree 1. There are therefore only three possible cases:
		\begin{itemize}
			\item if $w$ has degree 3 then $\inv(G)\approx7.80<8.12\approx\inv(G_{7,7})$ (see Figure \ref{figExemples});
			\item if $w$ has degree 4 and a neighbor of degree 1, then $\inv(G)\approx9.12<9.24\approx\inv(H_{8,8})$ (see Figure \ref{fign2n3});
			\item if $w$ has degree 4 and a second neighbor of degree 4, then 
			$\inv(G)\approx12.62<12.78\approx\inv(G_{11,11})$ (see Figure \ref{figExemples});
		\end{itemize}
		In all cases $G$ is not a connected  extremal  chemical graph, a contradiction.
		\item Assume that $u$ and $v$ have no common neighbor. Let $x \neq v$ (resp. $y \neq u$) be the second neighbor of $u$ (resp. $v$). Let $G'$ be the graph obtained from $G$ by replacing $xu$ with $xv$. Then, with $i=d_x$ and $j=d_y$, we have
		\vspace{-0.3cm}\begin{align*}
		\inv(G') - \inv(G) &\ge c_{1,3} - c_{2,2} + \min_{i = 1, \ldots 4} (c_{3,i} - c_{2,i}) + \min_{j = 1, \ldots 4} (c_{3,j} - c_{2,j})\\& \approx 0.0541 > 0.
		\end{align*}
		Hence, $G$ is not extremal, a contradiction.\qedhere
	\end{itemize}
\end{proof}

\begin{lemm}[\bf 1,2]
	The only connected extremal  chemical graphs $G$ with $x_{1,2}(G)>0$ are $H_{3,2}$ and $H_{6,5}$.
\end{lemm}

\begin{proof}
	Let $G$ be a connected extremal  chemical graph with two adjacent vertices $u$ and $v$ such that $d_u=1$ and $d_v=2$, and let $w$ be the other neighbor of $v$. If $w$ has degree 1 then $G \simeq H_{3,2}$. We know from Lemma (2,2) that $w$ does not have degree 2, and from Lemma \ref{lem_rotation}(a) that $w$ does not have degree 3. Hence, $d_w=4$. If $w$ has three neighbors of degree 1, then $G \simeq H_{6,5}$; otherwise, $w$ has a neighbor $x\neq v$ of degree at least 2 and Lemma \ref{lem_2opt} with the partial chain $u,v,w,x$ contradicts the fact that $G$ is a connected extremal chemical graph.   
\end{proof}

\begin{lemm}[\bf 3,3]
	The only connected extremal  chemical graphs $G$ with\linebreak $x_{3,3}(G)>0$ are $H_{4,5}$, $H_{4,6}$, $H_{5,8}$ and $H_{6,8}$.
\end{lemm}
\begin{proof}
	Let $G$ be a connected extremal = chemical graph with  two adjacent vertices $u$ and $v$ of degree 3. Let us first show that all vertices  $x\neq u,v$ are either adjacent to both $u$ and $v$, or to neither.
	If this is not the case then, we consider two cases.
	\begin{itemize}[itemsep=-10pt]
		\item If $u$ and $v$ have a common neighbor $w$, then let $x\neq v$ be a vertex adjacent to $u$ but not to $v$, and let $y\neq u$ be a vertex adjacent to $v$ but not to $u$. Then $G$ is not extremal. Indeed, let $G'$ be the graph obtained from $G$ by replacing $xu$ with $xv$. Then, with $i=d_x$, $j=d_y$ and $k=d_w$, we have 
		\begin{eqnarray*}
			\inv(G') - \inv(G) &{\geq}&  c_{2,4} - c_{3,3}+ \min_{i = 1,2,3,4} (c_{4,i} {-} c_{3,i})\\&& + \min_{j = 1,2,3,4} (c_{4,j} {-} c_{3,j})+\min_{k = 2,3,4} (c_{4,k} {+} c_{2,k}{-}2c_{3,k})\\
			& {\approx}& 0.0593 > 0.
		\end{eqnarray*}
		\item If $u$ and $v$ have no common neighbor, then there are two possible cases.
		\begin{itemize} 
			\item if all neighbors $x\neq u,v$ of $u$ and $v$ have degree 1, then
			$G$ is not extremal since $\inv(G)\approx5.62<5.87\approx\inv(H_{6,5})$ (see Figure \ref{figxij});
			\item if at least one of $u,v$, say $u$, has a neighbor $x\neq v$ of degree at least 2, then $G$ is not extremal. Indeed, let $y$ be the other neighbor of $u$ and let $G'$ be the graph obtained from $G$ by replacing $yu$ with $yv$. Then, with $i=d_x$, $j=d_y$ and $k=d_z$, where $z$ is any neighbor of $v$ other than $u$, we have
			\begin{eqnarray*}
				\inv(G') {-} \inv(G) &\geq& c_{2,4}{-}c_{3,3} +\min_{i = 2,3,4}(c_{2,i} {-} c_{3,i})\\&&+\min_{j = 1,2,3,4}(c_{4,j}{-} c_{3,j}) +2\min_{k = 1,2,3,4}(c_{4,k}{-} c_{3,k})\\
				& \approx& 0.0089 > 0.
			\end{eqnarray*}
		\end{itemize}
	\end{itemize}
	\vspace{-8pt}
	Hence, let $x,y$ the the two common neighbors of $u$ and $v$. We next show that at least one of $x,y$ has degree 4, or $G$ is $H_{4,5}$ or $H_{4,6}$.
	\begin{itemize} [itemsep=0pt]
		\item If $x$ has degree 3, then $x$ is adjacent to $y$ since $xu$ is an edge linking two vertices of degree 3 and we have seen that this implies that $y$ cannot be adjacent to exactly one of $u,x$. Hence, either $G \simeq H_{4,6}$, or $y$ has degree 4. 
		\item If $x$ has degree 2, then we know from the previous case that $y$ is of degree 2 or 4. Hence, either $G \simeq H_{4,5}$, or $y$ has degree 4. 
	\end{itemize}
	So, without loss of generality, assume $d_x=4$. 
	Let 
	$W=V\setminus \{u,v,x,y\}$ where $V$ is the vertex set of $G$. Assume that $W$ contains at least one vertex $w$ in $W$ of degree at least 3.
	\begin{itemize}[itemsep=0pt]
		\item If $d_w=4$, then let $z$ be a vertex adjacent to $w$ but not to $x$ and let $G'$ be the graph obtained from $G$ by replacing the edges $ux$ and $wz$ by $uw$ and $xz$. Clearly, $\inv(G')=\inv(G)$, which means that $G'$ is also a connected extremal  chemical graph. But $u$ and $v$ are two adjacent vertices in $G'$, and they are both of degree 3, while $x$ is adjacent to $u$ but not to $v$. We have shown above that this is impossible.
		\item If $d_w=3$, then 
		$w$ is adjacent to $x$. Indeed, if this is not the case, then let $z$ be any vertex in $W$ adjacent to $w$ and let $G'$ be the graph obtained from $G$ by replacing the edges $ux$ and $wz$ by $uz$ and $xw$. Clearly, $\inv(G')=\inv(G)$ and $u,v$ are two adjacent vertices of degree 3 in $G'$ with $z$ adjacent to $u$ but not to $v$, a contradiction.
		Moreover, all neighbors of $w$ in $W$ are adjacent to $x$. Indeed, assume that a vertex $z\in W$ is adjacent to $w$ but not to $x$. Then $d_z\leq 2$ (since vertices of degree 3 in $W$ are adjacent to $x$ and no vertex in $W$ has degree 4), and Lemma \ref{lem_2opt} with the partial chain $z,w,x,u$ shows that $G$ is not extremal, a contradiction.
		In summary, we have shown that $w$ can have only one neighbor in $W$ (else $x$ would be of degree at least 5), which means that $w$ is adjacent to $y$. We know from Lemma \ref{lem_rotation}(b) that the neighbor $z$ of $w$ in $W$ cannot be of degree 2  (since $u$ has degree 3 and is not adjacent to $w$). Hence, $d_z=3$, which means that $w,x$ and $y$ are its three neighbors (as $z,x,y$ are the 3 neighbors of $w$). Hence, $\inv(G)\approx10.08<10.28\approx\inv(G_{6,10})$ (see Figure \ref{figExemples}), a contradiction. 
	\end{itemize}
	
	\noindent Hence, all vertices in $W$ have degree 1 or 2. At least one vertex in $W$ has degree 2, else
	\begin{itemize} 
		\item if $d_y=3$, then $x$ is adjacent to $y$ since $uy$ is an edge linking two vertices of degree 3 and we have seen that this implies that $x$ cannot be adjacent to exactly one of $u,y$. Hence, $\inv(G)\approx7.28<7.36\approx\inv(H_{5,7})$ (see Figure \ref{fign2n3});
		\item if $d_y=4$, then either $G \simeq H_{6,8}$, or $\inv(G)\approx10.04<10.12\approx\inv(G_{8,9})$ (see Figure \ref{figExemples}).
	\end{itemize}
	
	So let $z$ be a vertex of degree 2 in $W$. We know  from Lemmas (2,2) and (1,2) that $z$ is adjacent to $x$ and $y$, which implies that $y$ has degree 4, else $u$ and $y$ are two adjacent vertices of degree 3 and $z$ is adjacent to $y$ but not to $u$, which is impossible. There are therefore three possible cases:
	\begin{itemize}[itemsep=0pt]
		\item if  $W$ has two vertices of degree 2, then $\inv(G)\approx9.28<9.40\approx\inv(H_{6,9})$ (see Figure \ref{fign2n3});
		\item if $W$ has one vertex of degree 2 and $x$ is not adjacent to $y$, then $\inv(G)\approx9.66<9.78\approx\inv(G_{7,9})$ (see Figure \ref{figExemples});
		\item if $W$ has one vertex of degree 2 and $x$ is adjacent to $y$, then $G \simeq H_{5,8}$.\qedhere
	\end{itemize}
	
\end{proof}

\begin{lemm}[\bf 1,3]
	The only connected extremal  chemical graphs $G$ with\linebreak$x_{1,3}(G)>0$ are $H_{4,3}$, $H_{4,4}$, $H_{6,6}$, $H_{7,6}$ and $H_{10,9}$.
\end{lemm}

\begin{proof}
	Let $G$ be a connected extremal  chemical graph with two adjacent vertices $u$ and $v$ such that $d_u=1$ and $d_v=3$. If $v$ has a neighbor $w$ of degree 2, we know from Lemma \ref{lem_rotation}(a) that the second neighbor $x$ of $w$ is adjacent to $v$. Then, either $G \simeq H_{4,4}$ or it follows from Lemmas (2,2) and (3,3) that $d_x=4$. If $x$ has two neighbors of degree 1, then $G \simeq H_{6,6}$, else $x$ has a neighbor $y\neq v,w$ such that $d_y\geq 2$, and Lemma \ref{lem_2opt} with the partial chain $u,v,x,y$ shows that $G$ is not extremal, a contradiction.
	
	So assuming that $G$ is not $H_{4,4}$ or $H_{6,6}$, we know that no neighbor of $v$ has degree 2. It then follows from Lemma (3,3) that they all have degree 1 or 4. If $v$ has a neighbor $x$ of degree 4, then all neighbors $y \neq v$ of $x$ that are also not adjacent to $v$ have degree 1, else Lemma \ref{lem_2opt} with the partial chain $u,v,x,y$ shows that $G$ is not extremal. Hence there are only 4 cases:
	\begin{itemize}[itemsep=0pt]
		\item if all neighbors of $v$ have degree 1, then $G \simeq H_{4,3}$;
		\item if $v$ has only one neighbor of degree 4, then $G \simeq H_{7,6}$;
		\item if $v$ has two non-adjacent neighbors of degree 4, then  $G \simeq H_{10,9}$;
		\item if $v$ has two adjacent neighbors of degree 4, then 
		$\inv(G)\approx 9.18 < 9.24 \approx \inv(H_{8,8})$ (see Figure \ref{fign2n3}).\qedhere
	\end{itemize}
\end{proof}

\begin{lemm}[\bf 2,3]
	The only connected extremal  chemical graphs $G$ with\linebreak$x_{2,3}(G)>0$ are $H_{4,4}$, $H_{4,5}$, $H_{5,6}$ and $H_{6,6}$.
\end{lemm}

\begin{proof}
	Let $G$ be a connected extremal  chemical graph with two adjacent vertices $u$ and $v$ such that $d_u=2$ and $d_v=3$. We know from Lemma \ref{lem_rotation}(a) that the second neighbor $w$ of $u$ is adjacent to $v$. If $G$ is not equal to $H_{4,4}$ or $H_{4,5}$, it follows from Lemmas (2,2) and (3,3) that $w$ has degree 4. Also, it follows from Lemmas (1,3) and (3,3) that if $G$ is not $H_{6,6}$, then the third neighbor $x\neq u,w$ of $v$ has degree 2 or 4.
	\begin{itemize}
		\item If $d_x=2$, then $x$ is adjacent to $w$ since, by Lemma \ref{lem_rotation}(a), the second neighbor of $x$ must be adjacent to $v$. It follows from Lemma \ref{lem_rotation}(b) that all vertices other than $u,v,w,x$ have degree 1 or 4. Hence, the fourth neighbor $y\neq u,v,x$ of $w$ has degree 1 or 4. If $d_y=1$ then $G \simeq H_{5,6}$. If $d_y=4$, then there are two cases:
		\begin{itemize}
			\item if $y$ has 3 neighbors of degree 1, then 
			$\inv(G)\approx9.92<10.12\approx\inv(G_{8,9})$ (see Figure \ref{figExemples});
			\item if $y$ has a neighbor $z\neq w$ of degree 4, then 
			Lemma \ref{lem_2opt} with the partial chain $u,v,w,y,z$ shows that $G$ is not extremal, a contradiction.
		\end{itemize}
		\item If $d_x=4$, then let $y\neq v,w$ be a neighbor of $x$. It follows from Lemmas \ref{lem_rotation}(b) and \ref{lem_2opt} with the partial chain $u,v,x,y$
		that $d_y=1$. Let $G'$ be the graph obtained from $G$ by replacing the edges $uw$ and $xy$ by $ux$ and $wy$. Clearly, $\inv(G')=\inv(G)$, which means that $G'$ is also a connected  extremal  chemical graph and since $w$ is now the neighbor of $v$ of degree 4 that is not adjacent to $u$, this means that all neighbors of $w$ different from $v,x$ have degree 1. Hence, 
		\begin{itemize}
			\item if $w$ is adjacent to $x$ then $\inv(G)\approx8.85<8.86\approx\inv(H_{7,8})$ (see Figure \ref{fign2n3});
			\item if $w$ is not adjacent to $x$ then $\inv(G)\approx10.35<10.5\approx\inv(G_{9,9})$ (see Figure~\ref{figExemples}).\qedhere
		\end{itemize}
	\end{itemize} 
\end{proof}


\section{Characterization of extremal chemical graphs}\label{sec:MainTheorem}

In this section, we characterize extremal chemical graphs of order $n$ and size $m\geq n-1$. We first consider the connected extremal chemical graphs, and then the non-connected ones. We conclude the section with a property of extremal  chemical graphs or order $n$ and size $m\leq n-2$. 

We start with the proof of Theorem \ref{thm1} that states that a connected extremal chemical graph of order $n$ and size $m$ necessarily belongs to $\mathcal{G}_{n,m}$, except for 22 pairs $(n,m)$.\\

\noindent {\em Proof of Theorem \ref{thm1}.} Observe first that if $(t_1,t_2,t_3,t_4)$ is a quadruplet in $T_{n,m}$ with $t_2+t_3>1$. Then there is $(s_1,s_2,s_3,s_4)\in T_{n,m}$ such that $s_2+s_3<t_1+t_2$ and $f(s_1,x_2,s_3,s_4)>f(t_1,t_2,t_3,t_4)$. Indeed:
\begin{itemize}
	\item If $t_2\geq 2$ then set $s_1=t_1+1$, $s_2=t_2-2$, $s_3=t_3+1$ and $s_4=t_4$. We have, $\sum_{i=1}^4s_i=\sum_{i=1}^4t_i=n$ and
	$\sum_{i=1}^4is_i=\sum_{i=1}^4it_i=2m$, which means that $(s_1,s_2,s_3,s_4)\in T_{n,m}$. Moreover, $$f(s_1,s_2,s_3,s_4)-f(t_1,t_2,t_3,t_4)=c_{1,4}-4c_{2,4}+3c_{3,4}\approx 0.0384 > 0.$$
	\item if $t_3\geq 2$ then set $s_1=t_1$, $s_2=t_2+1$, $s_3=t_3-2$ and $s_4=t_4+1$. We have 
	$\sum_{i=1}^4s_i=\sum_{i=1}^4t_i=n$ and
	$\sum_{i=1}^4is_i=\sum_{i=1}^4it_i=2m$, which means that $(s_1,s_2,s_3,s_4)\in T_{n,m}$. Moreover, $$f(s_1,s_2,s_3,s_4)-f(t_1,t_2,t_3,t_4)=2c_{2,4}-6c_{3,4}+4\approx 0.0591 > 0.$$
	\item if $t_2\geq 1$ and $t_3\geq 1$, then set $s_1=t_1+1$, $s_2=t_2-1$, $s_3=t_3-1$ and $s_4=t_4+1$. Hence, 
	$\sum_{i=1}^4s_i=\sum_{i=1}^4t_i=n$ and
	$\sum_{i=1}^4is_i=\sum_{i=1}^4it_i=2m$, which means that $(s_1,s_2,s_3,s_4)\in T_{n,m}$. Moreover, $$f(s_1,s_2,s_3,s_4)-f(t_1,t_2,t_3,t_4)=c_{1,4}-2c_{2,4}-3c_{3,4}+4\approx 0.0975 > 0.$$
\end{itemize}	

Note that if $s_2+s_3>1$, then we can repeat the same reasoning. We can therefore conclude that if $(t_1,t_2,t_3,t_4)$ is a quadruplet in $T_{n,m}$ with $t_2+t_3>1$, then there is $(s_1,s_2,s_3,s_4)\in T_{n,m}$ such that $s_2+s_3\leq 1$ and $f(s_1,s_2,s_3,s_4)>f(t_1,t_2,t_3,t_4)$.

So let $G$ be a connected extremal chemical graph of order $n$ and size $m$, and suppose that $G$ is not one of the 22 graphs of Figures \ref{fign2n3} and \ref{figxij}. It follows from the lemmas of the previous section that all edges in $G$ have at least one endpoint of degree 4. Since $n_0(G)=0$ (else $G\simeq H_{1,0}$), we have $\inv(G)=f(n_1(G),n_2(G),n_3(G),n_4(G))$. We have shown above that if $n_2(G)+n_3(G)>1$, then there is a quadruplet  $(s_1,s_2,s_3,s_4)$ in $T_{n,m}$ such that $s_2+s_3\leq 1$ and 
$f(s_1,s_2,s_3,s_4)>f(n_1(G),n_2(G),n_3(G),n_4(G))$. Hence, if $n_2(G)+n_3(G)>1$, then it follows from Property \ref{prop:existence} that there is a connected chemical graph $G'$ in $\mathcal{G}_{n,m}$ with $\inv(G')=f(s_1,s_2,s_3,s_4)>f(n_1(G),n_2(G),n_3(G),n_4(G))=\inv(G)$, a contradiction. We can therefore conclude that $n_2(G)+n_3(G)\leq 1$, which implies that $G$ belongs to $\mathcal{G}_{n,m}$.
\qed\\

It follows from Theorem \ref{thm1} and Corollary \ref{cor2} that if $1\leq n-1\leq m$ and $(n,m)$ is not a pair for which there is a graph $H_{n,m}$ in Figure \ref{fign2n3} or \ref{figxij} and if there exists a connected extremal graph of order $n$ and size $m$, then all graphs in $\mathcal{G}_{n,m}$ are extremal and their arithmetic-geometric index is easy to compute since we know the number of edges with endpoints of degree $i$ and $j$ for all $1\leq i\leq j\leq 4$. We can therefore state the following corollary.  

\begin{cor}\label{corbound}
	Let $G$ be a connected extremal chemical graph. If $G$ is not one of the 22 graphs $H_{n,m}$ in Figure \ref{fign2n3}, then
	$\emph\inv(G)=U\!B_{n,m}$, where    $$U\!B_{n,m}=\frac{2n+5m}{6}+\left\{\begin{array}{ll}0& \mbox{\emph{if} } 2m-n\equiv 0 \bmod{3},\\
	\frac{3}{\sqrt{2}}-\frac{13}{6}&\mbox{\emph{if} } 2m-n\equiv 1 \bmod{3},\\\frac{21}{4\sqrt{3}}-\frac{37}{12}&\mbox{\emph{if} } 2m-n\equiv 2 \bmod{3}.\end{array}\right.$$
\end{cor}
\begin{proof}$ $
	Theorem \ref{thm1} shows that $G$ belongs to $\mathcal{G}_{n,m}$. Let us compute $\inv(G)$.
	\begin{itemize} [leftmargin=7mm]
		\item If $2m-n\equiv 0 \bmod{3}$, then $4n-2m\equiv 0 \bmod{3}$, which means that $n_1(G){=}\frac{4n{-}2m}{3}$, $n_2(G){=}n_3(G){=}0$ and $n_4(G){=}\frac{2m{-}n}{3}$. Hence,
		$\inv(G)=\tfrac{3}{4}\tfrac{4n-2m}{3}+2\tfrac{2m-n}{3}=\tfrac{2n+5m}{6}$.
		\item If $2m{-}n\equiv 1 \bmod{3}$, then $4n{-}2m\equiv 2 \bmod{3}$, which means that $n_1(G){=}\frac{4n{-}2m{-}2}{3}$, $n_2(G){=}1, n_3(G){=}0$ and $n_4(G){=}\frac{2m{-}n{-}1}{3}$. Hence,
		$\inv(G)=\tfrac{3}{4}\tfrac{4n-2m-2}{3}+(\tfrac{3}{\sqrt{2}}-1)+2\tfrac{2m-n-1}{3}=\tfrac{2n+5m-13}{6}+\tfrac{3}{\sqrt{2}}$.
		\item If $2m{-}n\equiv 2 \bmod{3}$, then $4n{-}2m\equiv 1 \bmod{3}$, which means that $n_1(G){=}\frac{4n{-}2m{-}1}{3}$, $n_2(G){=}0, n_3(G){=}1$ and $n_4(G){=}\frac{2m{-}n{-}2}{3}$. Hence,
		$\inv(G)=\tfrac{3}{4}\tfrac{4n-2m-1}{3}+(\tfrac{21}{4\sqrt{3}}-\tfrac{3}{2})+2\tfrac{2m-n-2}{3}=\tfrac{2n+5m}{6}+\tfrac{21}{4\sqrt{3}}-\tfrac{37}{12}$.
		\qedhere
	\end{itemize}
\end{proof}

As shown in Table \ref{Table_uB}, if $(n,m)$ is one of the pairs for which there is a graph $H_{n,m}$ in Figure \ref{fign2n3} or \ref{figxij}, then $\inv(H_{n,m})<U\!B_{n,m}$. Hence, the connected graphs in $\mathcal{G}_{n,m}$ are the only connected chemical graphs $G$ of order $n$ and size $m$ with $\inv(G)=U\!B_{n,m}$. The sharp upper bound $\inv(H_{n,m})$ for the 22 pairs $(n,m)$ that are exceptions is slightly smaller than $U\!B_{n,m}$. We give in Table \ref{Table_uB} the values of this upper bound as well as the differences between $U\!B_{n,m}$ and $\inv(H_{n,m})$. We observe that the  largest difference is $\frac{1}{2}$ while the smallest is approximately equal to 0.0384.

\begin{table}[!htb]
	\centering
	\caption{Sharp upper bound $\inv(H_{n,m})$ on the arithmetic-geometric index of graphs of order $n$ and size $m$, for the 22 pairs $(n,m)$ not included in Corollary \ref{corbound}, and difference  with $U\!B_{n,m}$.}
	\label{Table_uB}
	\vspace{0.4cm}\begin{tabular}{|c|c||c|ccc|}
		\multicolumn{1}{c}{$n$}&\multicolumn{1}{c}{$m$}&\multicolumn{1}{c}{$\inv(H_{n,m})$}&\multicolumn{3}{c}{$U\!B_{n,m}-\inv(H_{n,m})$}\\
		\cline{1-6}1 &   0   &0&   $-\frac{11}{4}+\frac{21}{4\sqrt{3}}$   &$\approx$&   0.2811\\[2pt]
		2 &   1   &1&   $\frac{1}{2}$   &$\approx$& 0.5000\\[2pt]
		3 &   2   &$\frac{3}{\sqrt{2}}$&   $\frac{1}{2}$   &$\approx$&   0.5000\\[2pt]
		3 &   3   &3&   $\frac{1}{2}$   &$\approx$&   0.5000\\[2pt]
		4 &   3   &$\frac{6}{\sqrt{3}}$&   $\frac{3}{4}-\frac{3}{4\sqrt{3}}$   &$\approx$&   0.3170\\[2pt]
		4 &   4   &$1+\frac{2}{\sqrt{3}}+\frac{5}{\sqrt{6}}$&   $\frac{3}{2}+\frac{3}{\sqrt{2}}-\frac{2}{\sqrt{3}}-\frac{5}{\sqrt{6}}$   &$\approx$&   0.4254\\[2pt]
		4 &   5   &$1+\frac{10}{\sqrt{6}}$&   $\frac{9}{2}-\frac{10}{\sqrt{6}}$   &$\approx$&   0.4175\\[2pt]
		4 &   6   &6&   $-\frac{11}{4}+\frac{21}{4\sqrt{3}}$   &$\approx$&    0.2811\\[2pt]
		5 &   5   &$\frac{7}{2}+\frac{3}{\sqrt{2}}$&   $-\frac{3}{4}-\frac{3}{\sqrt{2}}+\frac{21}{4\sqrt{3}}$   &$\approx$&  0.1598\\[2pt]
		5 &   6   &$\frac{5}{4}+\frac{3}{\sqrt{2}}+\frac{7}{4\sqrt{3}}+\frac{5}{\sqrt{6}}$&   $\frac{13}{4}-\frac{7}{4\sqrt{3}}-\frac{5}{\sqrt{6}}$   &$\approx$&  0.1984\\[2pt]
		5 &   7   &$1+\frac{9}{\sqrt{2}}$&   $\frac{13}{2}-\frac{9}{\sqrt{2}}$   &$\approx$&   0.1360\\[2pt]
		5 &   8   &$2+\frac{3}{\sqrt{2}}+\frac{7}{\sqrt{3}}$&   $\frac{13}{4}-\frac{3}{\sqrt{2}}-\frac{7}{4\sqrt{3}}$   &$\approx$&   0.1183\\[2pt]
		5 &   9   &$3+\frac{21}{2\sqrt{3}}$&   $4+\frac{3}{\sqrt{2}}-\frac{21}{2\sqrt{3}}$ &$\approx$&   0.0591\\[2pt]
		6 &   5   &$\frac{15}{4}+\frac{3}{\sqrt{2}}$&   $\frac{1}{4}$    &$\approx$&  0.2500\\[2pt]
		6 &   6   &$\frac{5}{2}+\frac{3}{2\sqrt{2}}+\frac{15}{4\sqrt{3}}+\frac{5}{2\sqrt{6}}$&   $\frac{9}{2}-\frac{3}{2\sqrt{2}}-\frac{15}{4\sqrt{3}}-\frac{5}{2\sqrt{6}}$   &$\approx$&  0.2537\\[2pt]
		6 &   7   &$\frac{7}{2}+\frac{6}{\sqrt{2}}$&   $\frac{5}{4}-\frac{6}{\sqrt{2}}+\frac{21}{4\sqrt{3}}$   &$\approx$&   0.0384\\[2pt]
		6 &   8   &$\frac{9}{2}+\frac{7}{\sqrt{3}}$&   $2+\frac{3}{\sqrt{2}}-\frac{7}{\sqrt{3}}$   &$\approx$&   0.0799\\[2pt]
		6 &   9   &$\frac{17}{4}+\frac{3}{\sqrt{2}}+\frac{21}{4\sqrt{3}}$&   $\frac{21}{4}-\frac{3}{\sqrt{2}}-\frac{21}{4\sqrt{3}}$   &$\approx$&  0.0976\\[2pt]
		7 &   6   &$\frac{15}{4}+\frac{4}{\sqrt{3}}+\frac{7}{4\sqrt{3}}$&   $\frac{1}{2}-\frac{1}{2\sqrt{3}}$   &$\approx$&   0.2113\\[2pt]
		7 &   8   &$\frac{5}{2}+\frac{9}{\sqrt{2}}$&   $\frac{13}{2}-\frac{9}{\sqrt{2}}$   &$\approx$&   0.1360\\[2pt]
		8 &   8   &$5+\frac{6}{\sqrt{2}}$&   $\frac{5}{4}-\frac{6}{\sqrt{2}} +\frac{21}{4\sqrt{3}}$  &$\approx$&    0.0384 \\[2pt]
		10 &   9   &$\frac{15}{2}+\frac{11}{2\sqrt{3}}$&   $\frac{1}{4}-\frac{1}{4\sqrt{3}}$   &$\approx$&  0.1057\\[2pt]
		\hline
		
	\end{tabular}
\end{table}

When $m=n-1$, Corollary \ref{corbound} gives an upper bound for chemical trees. More precisely, if $T$ is a chemical tree of order $n$, then 
$$\inv(T)\leq U\!B_{n,n-1}=\frac{7n-5}{6}+\left\{\begin{array}{ll}0& \mbox{\emph{if} } n\equiv 2 \bmod{3},\\
\frac{3}{\sqrt{2}}-\frac{13}{6}&\mbox{\emph{if} } n\equiv 0 \bmod{3},\\\frac{21}{4\sqrt{3}}-\frac{37}{12}&\mbox{\emph{if} } n\equiv 1 \bmod{3}.\end{array}\right.$$
and this bound is reached for all $n$, except for $n=1,2,3,4,6,7,10$ since $H_{1,0},H_{2,1},H_{3,2},H_{4,3},H_{6,5},H_{7,6}$ and $H_{10,9}$ appear in Figure \ref{figxij}. The same upper bound is given in \cite{BIBLIO5}, but the authors did not mention the  7 exceptions. For example, they state that when $n\equiv 1 \bmod{3}$, there are $\frac{n-1}{3}-1$ vertices of degree 4, and one vertex of degree 3 that must be adjacent to vertices of degree 4. This is clearly impossible for $n=1,4,7$ and $10$.\\

We now show that if we remove the constraint that extremal chemical graphs must be connected, then no better value of $\inv$ can be obtained.
\begin{thm}\label{thm:nonconnected}
	Let $G$ be a non-connected  chemical graph of order $n$ and size $m\geq n-1$. If $G$ is extremal, then $G$ belongs to $\mathcal{G}_{n,m}$.  
\end{thm}
\begin{proof}

	Assume that the theorem is not valid and let $G$ be a non-connected extremal chemical graph of order $n$ and size $m\geq n-1$ that is a counterexample with the smallest number of connected components. It follows from Property \ref{22exceptions} that $(n,m)$ is not one of the 22 pairs for which there is a graph $H_{n,m}$ in Figure \ref{fign2n3} or \ref{figxij}.  
	Let $G_1,\ldots,G_k$ ($k\geq 2$) be the connected components of $G$, and let $N_i$ and $M_i$ be the order and the size of $G_i$, respectively. Clearly, $\inv(G)=\sum_{i=1}^k\inv(G_i)$.
	Hence, since $G$ is extremal, every $G_i$ is a connected  extremal graph of order $N_i$ and size $M_i$. 
	At least one $G_i$, say $G_1$, contains a cycle  $C$.
	If $C$ contains an edge $xy$ with $d_x=4$ and $d_y\geq 3$ then:
	\begin{itemize}
		\item if $G_2$ contains only one vertex $z$ then let $G'$ be the graph obtained from $G$ by replacing the edge $xy$ by the edge $xz$. Since $y$ belongs to a cycle, at least one of its neighbors $z \neq x$ has degree at least 2. Hence, with $i=d_y$, $j=d_z$ and $k=d_u$, where $u$ is any neighbor of $y$ other than $x$ and $w$, we have
		\begin{eqnarray*}
			\inv(G') {-} \inv(G) &\geq& c_{1,4} +\min_{i = 3,4} \bigl(
			\min_{j = 2,3,4}(c_{i-1,j} {-} c_{i,j})-c_{4,i}\bigr)\\&&+\min_{i = 3,4} \bigl((i{-}2)\min_{k = 1,2,3,4}(c_{i-1,k} {-} c_{i,k})\bigr)\\
			& \approx& 0.0193 > 0.
		\end{eqnarray*}
		Hence $G$ is not extremal, a contradiction.
		\item if $G_2$ contains at least two vertices, then consider any edge $zw$ in $G_2$ and assume without loss of generality that $d_z\leq d_w$. Let $G'$ be the graph obtained from $G$ by replacing the edges $xy$ and $zw$ by the edges $xz$ and $yw$. Then, with $i=d_y$, $j=d_z$ and $k=d_w$, we have
		$$\inv(G') {-} \inv(G)\geq \min_{i = 3,4}\;\min_{j = 1,2,3,4}\;\min_{k = j,\ldots,4}(c_{4,j}+c_{i,k}-c_{4,i}-c_{j,k})=0.$$
		Moreover, all cases where $\inv(G')=\inv(G)$ have $d_x=d_w$ or/and $d_y=d_z$. Hence, $G'$ has a smaller number of connected components than $G$, while $n_i(G)=n_i(G')$ for  $0\leq i\leq 4$ and $x_{i,j}(G)=x_{i,j}(G')$ for $1\leq i\leq j\leq 4$. It follows that $G'$ is also extremal, and either both of $G$ and $G'$ belong to $\mathcal{G}_{n,m}$, or none of them. If $G'$ is connected, then we know from Theorem \ref{thm1} that $G'$ (and hence also $G$) belongs to $\mathcal{G}_{n,m}$, which means that $G$ is not a counterexample to the theorem, a contradiction. If $G'$ is not connected, then $G$ is not a counterexample to the theorem with the smallest number of connected components, a contradiction.
	\end{itemize}

	Note that if $G_1$ belongs to $\mathcal{G}_{N_1,M_1}$, 
	then the cycle $C$ contains two adjacent vertices of degree 4 (since vertices of degree 1 do not belong to a cycle and there is at most one vertex of degree 2 or 3 in $G_1$). Also, the 8 graphs $H_{5,6}$, $H_{5,7}$, $H_{5,8}$, $H_{5,9}$, $H_{6,6}$, $H_{6,7}$, $H_{6,8}$, $H_{6,9}$, in Figures \ref{fign2n3} and \ref{figxij} which have a cycle and an edge linking a vertex of degree 4 to a vertex of degree at least 3, have such an edge in a cycle.  Hence, $x_{3,4}(G_i)+x_{4,4}(G_i)=0$ for all connected components $G_i$ of $G$ with a cycle.
	
	Suppose now that $G_2$ contains an edge $xy$ with $d_x=4$ and $d_y\geq 3$. Consider any edge $zw$ on $C$ and assume without loss of generality that $d_z\leq d_w$.
	Let $G'$ be the graph obtained from $G$ by replacing the edges $xy$ and $zw$ by the edges $xz$ and $yw$. Then, with $i=d_y$, $j=d_z$ and $k=d_w$, we have
	$$\inv(G') {-} \inv(G)\geq \min_{i = 3,4}\;\min_{j = 2,3,4}\;\min_{k = j,\ldots,4}(c_{4,j}+c_{i,k}-c_{4,i}-c_{j,k})=0.$$
	As above, the only cases where $\inv(G')=\inv(G)$ have $d_x=d_w$ or/and $d_y=d_z$. Hence, either $G$ is not a counterexample to the theorem, or it is not a counterexample with the smallest number of connected components, a contradiction.
	
	Hence, we know that $x_{3,4}(G)+x_{4,4}(G)=0$. We now prove that no connected component of $G$ can have more than 9 vertices. So assume $G$ has a connected component $H$ of order $N\geq 10$ and size $M$. We know from Theorem \ref{thm1} that there are two possible cases:
	\begin{itemize} 
		\item if $H$ is one of the 22 graphs in Figures \ref{fign2n3} and \ref{figxij}, then $H\simeq H_{10,9}$, which implies $x_{3,4}(G){\geq }x_{3,4}(H){=}2$, a contradiction.
		\item if $H$ belongs to $\mathcal{G}_{N,M}$, then  $n_3(H)=0$ else $x_{3,4}(G)\geq x_{3,4}(H)=3$, Hence, $$x_{4,4}(H)\geq  m-n_1(H)-2\geq \Bigl\lceil m-\frac{4n-2m}{3}-2\Bigr\rceil=\Bigl\lceil\frac{5m-4n-6}{3}\Bigr\rceil.$$
		The four cases here below show that $x_{3, 4}(G) + x_{4,4}(G)\geq x_{3, 4}(H) +  x_{4,4}(H)\geq 1$, a contradiction.
		\begin{itemize}
			\item If $M=N-1=9$ then $x_{3,4}(H)=2$. \item If $M=N-1=10$, then $x_{4,4}(H)=2$. \item If  $M=N-1\geq 11$, then $x_{4,4}(H)\geq \lceil \frac{N-11}{3}\rceil\geq 1$. \item If $M\geq N\geq 10$ then $x_{4,4}(H)\geq \frac{N-6}{3}>1$. 
		\end{itemize}
	\end{itemize}
	
	We thus know that $3\leq N_1\leq 9$, $N_1\leq M_1 \leq \min\{2N_1,\frac{N_1(N_1-1)}{2}\}$ and $G_1$ has no edge linking a vertex of degree 4 to a vertex of degree at least 3. It is easy to check that there are exactly 7 such graphs, namely, $H_{3,3}$, $H_{4,4}$, $H_{4,5}$, $H_{4,6}$, $H_{5,5}$, $H_{7,8}$ and $H_{8,8}$ (see Figures \ref{fign2n3} and \ref{figxij}). Also, $1\leq N_2\leq 9$, $N_2-1\leq M_2 \leq \min\{2N_2,\frac{N_2(N_2-1)}{2}\}$ and $G_2$ has no edge linking a vertex of degree 4 to a vertex of degree at least 3. There are only 14 such graphs, namely, the 7 graphs mentioned above, and $H_{1,0}$, $H_{2,1}$, $H_{3,2}$, $H_{4,3}$, $G_{5,4}$, $H_{6,5}$ and $G_{9,8}$ (see Figures \ref{figExemples}, \ref{fign2n3} and \ref{figxij}).
	
	Let $i_{n,m}$ be equal to $\inv(G)$, where $G$ is any connected extremal chemical graph of order $n$ and size $m$. It is easy to check by enumeration that $i_{N_1,M_1}+i_{N_2,M_2}< i_{N_1+N_2,M_1+M_2}$ for the 7 pairs $(N_1,M_1)$ and the 14 pairs $(N_2,M_2)$. Hence by removing $G_1$ and $G_2$ and replacing these two connected components of $G$ by a connected extremal chemical graph of order $N_1+N_2$ and size $M_1+M_2$, one gets a graph $G'$ with $\inv(G)< \inv(G')$, which means that $G$ is not extremal, a contradiction.
\end{proof}

Corollary \ref{cor2} shows that all graphs in $\mathcal{G}_{n,m}$ have the same $\inv$ value, which means that they are all extremal if $(n,m)$ is not a pair appearing in Figures \ref{fign2n3} or \ref{figxij}. Hence, putting together Property \ref{22exceptions} and Theorems \ref{thm1} and \ref{thm:nonconnected}, we get the following characterization of extremal chemical graphs.\\

\begin{thm*}[Characterization of extremal chemical graphs]
	A chemical graph $G$ of order $n$ and size $m\geq n-1$ is extremal if and only if $G$ is one of the 22 graphs in Figures \ref{fign2n3} and \ref{figxij} or $G$ belongs to $\mathcal{G}_{n,m}$.
\end{thm*}

We indicate in Table \ref{TableConNonCon} the number of connected and non-connected extremal chemical graphs of order $n$ and size $m$ for $1\leq n \leq 14$ and $n-1\leq m\leq \min \{2n,\tfrac{n(n-1)}{2}\}$. For example, for $n=12$ and $m=11$, we see that there are exactly one connected and one non-connected extremal chemical graph and these two graphs are represented at the bottom of Figure \ref{figExemples}. The 22 pairs $(n,m)$ for which $H_{n,m}$ is the only extremal graph are shown in gray boxes. We observe that the number of connected extremal chemical graphs grows exponentially, but not in a monotonic way.

The proof of Theorem \ref{thm:nonconnected} shows that if a non-connected chemical graph $G$ contains a cycle, then there is a chemical graph $G'$ having fewer connected components than $G$ and such that $
\inv(G)\leq \inv(G')$. This leads to the following corollary.
\begin{cor}
	For all $n$ and $m$ with $0\leq m\leq n-2$ there is a chemical forest $G^*$ which is a disjoint union of extremal chemical trees and  such that $\emph\inv(G)\leq \emph\inv(G^*)$ for all chemical graphs $G$ of order $n$ and size $m$.
\end{cor}

\begin{table}[!htb]\footnotesize
	\centering
	\caption{Number of connected and non-connected extremal chemical graphs of order $n$ and size $m$ for $1\leq n \leq 14$ and $n-1\leq m\leq \min \{2n,\tfrac{n(n-1)}{2}\}$}
	\label{TableConNonCon}\addtolength{\tabcolsep}{-2pt}    
	\begin{tabular}{|r|r|rr|rr|rr|rr|rr|rr|rr|rr|rr|rr|rr|rr|rr|rr|}
		\multicolumn{2}{c}{}&\multicolumn{28}{c}{$n$}\\\cline{3-30}            \multicolumn{1}{l}{}&\multicolumn{1}{r}{}& \multicolumn{2}{|c}{1}& \multicolumn{2}{ c}{2}& \multicolumn{2}{c}{3}& \multicolumn{2}{c}{4}& \multicolumn{2}{c}{5}& \multicolumn{2}{ c}{6}& \multicolumn{2}{c}{7}& \multicolumn{2}{c}{8}& \multicolumn{2}{ c}{9}& \multicolumn{2}{ c}{10}& \multicolumn{2}{c}{11}& \multicolumn{2}{c}{12}& \multicolumn{2}{c}{13}& \multicolumn{2}{c|}{14}\\
		\cline{3-30}\multicolumn{1}{l}{$m$}&\multicolumn{1}{r}{}& \\
		\hhline{-~*{28}{-}}
		0	&&	\cellcolor{gray!60}1	&	\cellcolor{gray!60}0	&&	&&	&&	&&	&&	&&	&&	&&	&&	&&	&&	&&	&&	\\																					
		1	&&	&&	\cellcolor{gray!60}1	&	\cellcolor{gray!60}0	&&	&&	&&	&&	&&	&&	&&	&&	&&	&&	&&	&&	\\																					
		2	&&	&&	&&	\cellcolor{gray!60}1	&	\cellcolor{gray!60}0	&&	&&	&&	&&	&&	&&	&&	&&	&&	&&	&&	\\																					
		3	&&	&&	&&	\cellcolor{gray!60}1	&	\cellcolor{gray!60}0	&	\cellcolor{gray!60}1	&	\cellcolor{gray!60}0	&&	&&	&&	&&	&&	&&	&&	&&	&&	&&	\\																		
		4	&&	&&	&&	&&	\cellcolor{gray!60}1	&	\cellcolor{gray!60}0	&	1	&	0	&&	&&	&&	&&	&&	&&	&&	&&	&&	\\																		
		5	&&	&&	&&	&&	\cellcolor{gray!60}1	&	\cellcolor{gray!60}0	&\cellcolor{gray!60}1	&	\cellcolor{gray!60}0	&	\cellcolor{gray!60}1	&	\cellcolor{gray!60}0	&&	&&	&&	&&	&&	&&	&&	&&	\\															
		6	&&	&&	&&	&&	\cellcolor{gray!60}1	&	\cellcolor{gray!60}0	&	\cellcolor{gray!60}1	&	\cellcolor{gray!60}0	&	\cellcolor{gray!60}1	&	\cellcolor{gray!60}0	&	\cellcolor{gray!60}1	&	\cellcolor{gray!60}0	&&	&&	&&	&&	&&	&&	&&	\\												
		7	&&	&&	&&	&&	&&	\cellcolor{gray!60}1	&	\cellcolor{gray!60}0	&	\cellcolor{gray!60}1	&	\cellcolor{gray!60}0	&	1	&	0	&	1	&	0	&&	&&	&&	&&	&&	&&	\\												
		8	&&	&&	&&	&&	&&	\cellcolor{gray!60}1	&	\cellcolor{gray!60}0	&	\cellcolor{gray!60}1	&	\cellcolor{gray!60}0	&	\cellcolor{gray!60}1	&	\cellcolor{gray!60}0	&	\cellcolor{gray!60}1	&	\cellcolor{gray!60}0	&	1	&	0	&&	&&	&&	&&	&&	\\									
		9	&&	&&	&&	&&	&&	\cellcolor{gray!60}1	&	\cellcolor{gray!60}0	&	\cellcolor{gray!60}1	&	\cellcolor{gray!60}0	&	1	&	0	&	1	&	0	&	1	&	0	&	\cellcolor{gray!60}1	&	\cellcolor{gray!60}0	&&	&&	&&	&&	\\						
		10	&&	&&	&&	&&	&&	1	&	0	&	1	&	0	&	1	&	0	&	1	&	0	&	1	&	0	&	2	&	0	&	1	&	0	&&	&&	&&	\\			
		11	&&	&&	&&	&&	&&	&&	1	&	0	&	1	&	0	&	2	&	0	&	3	&	0	&	1	&	0	&	1	&	0	&	1	&	1	&&	&&	\\			
		12	&&	&&	&&	&&	&&	&&	1	&	0	&	2	&	0	&	4	&	0	&	2	&	0	&	4	&	0	&	6	&	0	&	2	&	0	&	1	&	0	&&	\\
		13	&&	&&	&&	&&	&&	&&	&&	2	&	0	&	3	&	0	&	10	&	0	&	12	&	0	&	4	&	0	&	5	&	1	&	7	&	1	&	2	&	1	\\
		14	&&	&&	&&	&&	&&	&&	&&	2	&	0	&	8	&	0	&	17	&	0	&	8	&	1	&	21	&	1	&	23	&	1	&	5	&	1	&	3	&	1	\\
		15	&&	&&	&&	&&	&&	&&	&&	&&	7	&	0	&	9	&	0	&	47	&	0	&	58	&	1	&	14	&	1	&	27	&	2	&	27	&	3	\\			
		16	&&	&&	&&	&&	&&	&&	&&	&&	6	&	0	&	37	&	0	&	77	&	0	&	31	&	1	&	113	&	2	&	111	&	4	&	18	&	2	\\			
		17	&&	&&	&&	&&	&&	&&	&&	&&	&&	28	&	0	&	35	&	0	&	249	&	0	&	303	&	3	&	59	&	4	&	159	&	11	\\						
		18	&&	&&	&&	&&	&&	&&	&&	&&	&&	16	&	0	&	198	&	0	&	399	&	0	&	134	&	2	&	684	&	8	&	625	&	20	\\						
		19	&&	&&	&&	&&	&&	&&	&&	&&	&&	&&	126	&	0	&	154	&	0	&	1550	&	1	&	1786	&	9	&	298	&	11	\\									
		20	&&	&&	&&	&&	&&	&&	&&	&&	&&	&&	59	&	1	&	1246	&	1	&	2395	&	1	&	707	&	7	&	4620	&	40	\\									
		21	&&	&&	&&	&&	&&	&&	&&	&&	&&	&&	&&	719	&	1	&	845	&	1	&	10801	&	4	&	11855	&	36	\\												
		22	&&	&&	&&	&&	&&	&&	&&	&&	&&	&&	&&	265	&	1	&	8789	&	3	&	16433	&	6	&	4399	&	20	\\												
		23	&&	&&	&&	&&	&&	&&	&&	&&	&&	&&	&&	&&	4721	&	3	&	5440	&	4	&	83399	&	19	\\															
		24	&&	&&	&&	&&	&&	&&	&&	&&	&&	&&	&&	&&	1544	&	3	&	68804	&	12	&	125829	&	28	\\															
		25	&&	&&	&&	&&	&&	&&	&&	&&	&&	&&	&&	&&	&&	35678	&	11	&	40399	&	14	\\																		
		26	&&	&&	&&	&&	&&	&&	&&	&&	&&	&&	&&	&&	&&	10778	&	8	&	590342	&	55	\\																		
		27	&&	&&	&&	&&	&&	&&	&&	&&	&&	&&	&&	&&	&&	&&	300361	&	45	\\																					
		28	&&	&&	&&	&&	&&	&&	&&	&&	&&	&&	&&	&&	&&	&&	88168	&	25	\\																					
		\cline{1-1}\cline{3-30}
	\end{tabular}
\end{table}

\section{Conclusion}

We have determined a sharp upper bound on the value of the arithmetic-geometric index of chemical graphs of order $n$ and size $m\geq n-1$, and we have characterized the chemical graphs that reach the bound.
This allows, for example, to characterize extremal chemical trees as well as extremal unicyclic or bicyclic chemical graphs. For $m\leq n-2$, we have shown that there is an extremal chemical graph or order $n$ and size $m$ which is a disjoint union of extremal chemical trees.

\acknowledgment{The authors would like to thank Pierre Hauweele for his help.}

\singlespacing

\end{document}